\documentclass[10pt]{amsart}

\usepackage{amssymb, amsthm, amsmath, graphicx, yfonts, wasysym, appendix, url, verbatim}
\usepackage[margin=4cm]{geometry}
\usepackage{rotating}
\usepackage{amsbsy,enumerate}
\usepackage{graphicx}
\usepackage{ccaption}
\usepackage{xcolor}
\usepackage{times}
\usepackage{wrapfig}
\usepackage{amsmath, amssymb, amsthm} 

\newcommand{\IP}[2]{\left< #1 , #2 \right>}

\newcommand{\R}{\ensuremath{\mathbb{R}}}
\newcommand{\D}{\ensuremath{\mathbb{D}}}

\renewcommand{\S}{\ensuremath{\mathbb{S}}}

\allowdisplaybreaks[4]

\newtheorem{thm}{Theorem}[section]
\newtheorem{cor}[thm]{Corollary}
\newtheorem{prop}[thm]{Proposition}
\newtheorem{lem}[thm]{Lemma}

\newtheorem*{uthm}{Theorem}

\theoremstyle{remark}
\newtheorem*{rmk}{Remark}

\title{Minimal hypersurfaces in the ball with free boundary}
\author{Glen Wheeler and Valentina-Mira Wheeler$^*$}
\thanks{*: Corresponding author.}
\address{Glen Wheeler\\
           Institute for Mathematics and its Applications \\
           University of Wollongong\\
           Northfields Avenue\\
           Wollongong, NSW, 2522, Australia\\
           email: glenw@uow.edu.au }
\address{ Valentina-Mira Wheeler \\
           Institute for Mathematics and its Applications \\
           University of Wollongong\\
           Northfields Avenue\\
           Wollongong, NSW, 2522, Australia\\
           email: vwheeler@uow.edu.au
           }

\begin{document}

\begin{abstract}
In this note we use the strong maximum principle and integral estimates prove
two results on minimal hypersurfaces $F:M^n\rightarrow\R^{n+1}$ with free
boundary on the standard unit sphere.
First we show that if $F$ is graphical with respect to any Killing field, then
$F(M^n)$ is a flat disk. This result is independent of the topology or number
or boundaries.  Second, if $M^n = \D^n$ is a disk, we show the supremum of the
curvature squared on the interior is bounded below by $n$ times the infimum of
the curvature squared on the boundary.
These may be combined the give an impression of the curvature of non-flat
minimal hyperdisks with free boundary.
\end{abstract}

\maketitle

\section{Introduction}

Recently, minimal surfaces with free boundary have received much attention.
A landmark result due to Nitsche is:

\begin{uthm}[Theorem 1 in \cite{nitsche1985}]
Let $F:\D^2 \rightarrow \R^{3}$ be a proper branched minimal immersion with free boundary on the standard unit sphere.
Then $F(\D^2)$ is a flat disk.
\end{uthm}

The proof exploits the Hopf differential via complex analysis.

There has been much work extending this result in various directions.
This activity has yielded some excellent results, as a small selection we refer
to
\cite{ambrozio2016gap,choe2011capillary,fraser2011first,fraser2012sharp,li2017gap,lopez2014capillarycone,lopez2014capillary,vogel1988uniqueness}.
Fraser-Schoen \cite{fraser2014uniqueness} made a recent influential
contribution, that includes an extension of Nitsche's Theorem to arbitrary
codimension.

In this note we study the higher dimensional analogue of this problem, for
minimal hypersurfaces with free boundary in the unit ball.
Although there is a wealth of knowledge available on the problem for $n=2$, in
the higher dimensional case results are much more scarce.
One reason for this is that incredibly powerful complex analytical techniques
that apply for surfaces do not seem to carry over to hypersurfaces.
Nevertheless, progress continues to be made: see Sargent \cite{sargent2016},
Ambrozio, Carlotto-Sharp \cite{ambrozio2016index}, Smith-Stern-Tran-Zhou
\cite{smith2017morse} and Tran \cite{tran2016index} for
some new index bounds for minimal hypersurfaces with free boundary,
Mondino-Spadaro \cite{monspad2017} for a new characterisation of free boundary
minimal submanifolds, and Li-Zhou \cite{GLZ2017,li2016min} for far-reaching
min-max and regularity theory, including an extension of the classical program
of Almgren \cite{almgren1,almgren2} (see Pitts \cite{pitts} and Schoen-Simon
\cite{schoen1981} for further classical theory) to the case of minimal
hypersurfaces with free boundary, for example.

Our first contribution is on the question of uniqueness of minimal embedded
$n$-disks. Note that this result is independent of topology. Under a
generalised graphicality condition, the only minimal hypersurfaces with free
boundary on the standard sphere are flat disks.

\begin{thm}[Uniqueness of $n$-dimensional graphical disks]
Let $F:M^n\rightarrow\R^{n+1}$ be a smooth Killing-graphical minimal
hypersurface with free boundary on $\S^n\subset\R^{n+1}$.  Then $M^n = \D^n$
and $M := F(M^n)$ is a standard flat disk.
\label{thmuniqueness}
\end{thm}

In the above statement, we use Killing-graphical to mean that the function
$s_V:M\rightarrow\R$ given by
\begin{equation}
\label{EQsV}
s_V(x) := \IP{\nu^M(x)}{V(x)}\,,
\end{equation}
where $\nu^M:M\rightarrow\R^{n+1}$ is a unit normal vector field along $F(M^n)$,
and $V:\R^{n+1}\rightarrow\R^{n+1}$ is a Killing field,
is strictly positive.

One may rephrase Theorem \ref{thmuniqueness} as:
If $s_V > 0$, then $F(M^n)$ is a flat disk.
We note that applying Theorem \ref{thmuniqueness} with $V$ a translation yields
Theorem \ref{thmuniqueness} as a corollary.
Theorem \ref{thmuniqueness} is proved in Section \ref{KillingSection}.

\begin{rmk}
Since one of the conclusions of Theorem \ref{thmuniqueness} (and not one of the
hypotheses) is topological, that is, that $M^n = \D^n$, we are able to use this
theorem in the contrapositive to obtain interesting topological lemmata.
The most general form of this is the following:

\begin{cor}
Suppose $M^n$ is not a disk.
Consider a smooth minimal hypersurface
$F:M^{n}\rightarrow\R^{n+1}$ with free boundary on $\S^{n}\subset \R^{n+1}$.
Then for every Killing field $V:\R^{n+1}\rightarrow\R^{n+1}$,
there exists a point $x\in M$ such that $\nu(x)\perp V(x)$.
\end{cor}

For example, this implies that on any free boundary minimal surface in the
topological class of the catenoid in $\R^3$, the functions $s_{V_i}$,
$i=1,\ldots,6$, attain at least one zero.
These kinds of results may be useful in understanding questions such as the
Fraser-Schoen conjecture.
\end{rmk}

The main result of Ambrozio-Nunes \cite{ambrozio2016gap} is that if $M^n = \D^n$ and
$|A|^2\IP{F}{\nu^M}^2 \le 2$, then $F(\D^n)$ is flat.
If we allow more freedom in the domain of $F$, the only other possibility is
that at some point $|A|^2\IP{F}{\nu^M}^2 = 2$ and $F(\S\times(a,b))$ is a
critical catenoid.
This result is special to the case of surfaces, but does indicate a kind of
`curvature gap' phenomenon at work.
Our second result moves also in this direction.

\begin{thm}[Curvature gap]
Let $F:\D^n\rightarrow\R^{n+1}$ be a smooth minimal immersed $n$-disk with free
boundary on $\S^n\subset\R^{n+1}$.
Suppose that $M := F(\D^n)$ is not a standard flat disk.
Then
\[
\bigg(\sup_{M} |A|^2\bigg)^2 > n\inf_{\partial M} |A|^2\,.
\]
\label{thmuniqueness2}
\end{thm}

\begin{rmk}
One has automatically that $\displaystyle \sup_{M} |A|^2 \ge \inf_{\partial M}
|A|^2$, and so our estimate above gives new information only when
$\displaystyle \inf_{\partial M} |A|^2 \in (0,4]$ for $n=2$ and $\in (0,2]$ for $n>2$.
We are not restricted here by dimension but we do require that the minimal
hypersurface is topologically a disk. This restriction is made clear in the
proof.
\end{rmk}

Finally, we wish to note that minimal hypersurfaces with free boundary on
$\S^n$ achieve equality in the isoperimetric inequality. The general result,
completed by Federer in \cite{federer2014geometric}, says that the area of an
$n$ dimensional minimal surface (any dimension) inside the unit sphere equals
$1/n$ times the integral over the boundary of the cosine of the angle it makes
with the radial direction. We would like to thank Professor Frank Morgan for
helping us locate the most general result of this kind.
Similar results have been obtained in Hanes \cite{hanes1972isoperimetric} and Brendle \cite{brendle2012}.

\begin{prop}[Isoperimetric equality, special case of Proposition 5.4.3(i) in \cite{federer2014geometric}]
Let $F:M^n\rightarrow\R^{n+1}$ be a smooth minimal immersed hypersurface with free
boundary on $\S^n\subset\R^{n+1}$.
Then
\begin{align*}
n|M| = |\partial M|,
\end{align*}
where $|M|$ and $|\partial M|$ are the volume of the hypersurface and the
volume of its boundary respectively.
\label{isoequality}
\end{prop}

\section{Setting}

Consider the standard unit sphere in Euclidean space $\S^n = \{x\in\R^{n+1}\,:\,|x|=1\}\subset
\R^{n+1}$.
We use $\nu^{\S^n}:\S^{n}\rightarrow\R^{n+1}$ to denote its outer
normal vectorfield.
Let $M^n$ be a smooth, orientable $n$-dimensional Hausdorff paracompact manifold with boundary $\partial M^n$.
Let $g$ be a Riemannian metric on $M^n$.
Set $M:=F(M^n)\subset{\R}^{n+1}$ where $F:M^n\rightarrow {\R}^{n+1}$ is a
smooth isometric immersion satisfying
\begin{align}
&{\partial} M \ \equiv \ F({\partial} M^n)\ =\ M \cap \S^n,\notag\\
&\IP{{\nu}^M}{{\nu}^{\S^n}}(F(p))\ =\ 0, \text{ }\forall
p~\in{\partial} M^n.
\label{initial}
\end{align}
Since $F$ is isometric, the Riemannian structure induced by the embedding $F$
is the same as that given by $g$, that is $(M^n, g) = (M^n, F^*\delta)$, where
$\delta$ is the standard metric on $\R^n$.

Let us denote by $A:T M\times T M \rightarrow \R$ the second fundamental form
of $M$ with components given by $h_{ij}$ where $1\leq i,j\leq n$ where $h_{ij}=
A(\tau_i,\tau_j)$ for two sections $\tau_i$ and $\tau_j$ in $TM$.
For $\S^n$ we have $A^{\S^n}:T \S^n \times T\S^n\rightarrow \R$ the
second fundamental form with components $h^{\S^n}_{ij}$ for $1\leq i,j\leq
n$.

\section{Auxiliary Equations}

Let us define the quantities we use in the proof of the uniqueness theorem.
Recall the function $s_V:M\rightarrow\R$ defined in \eqref{EQsV} above.
When $V$ is a translation, following \cite{ecker1989mce} we term $s_V$ the \emph{graph quantity}.
If, up to reparametrisation,
\[
F(p) = (p,u(p)) = p_ie_i + u(p)V\,,
\]
then one can relate the gradient of the associated scalar function $u$ to the
reciprocal of the graph quantity $s_V$. This implies that a lower bound on $s$
is equivalent to a gradient bound for $u$.
Throughout this section we assume that $F:M^n\rightarrow\R^{n+1}$ is a smooth
minimal immersed hypersurface.

\begin{lem} The quantity $s_V:M\rightarrow\R$
satisfies
\begin{align*}
{\Delta}^M s_V = - |A|^2 s_V\,.
\end{align*}
where $\Delta^M$ is the Laplace-Beltrami operator on $M$.
\label{sevolution}
\end{lem}
The squared reciprocal of $s_V$ denoted by $v_V^2=\frac{1}{s_V^2}$ also satisfies an
elliptic equation (see \cite{ecker1989mce}):
\begin{lem} The quantity $v_V^2:M\rightarrow\R$ satisfies
\begin{align*}
{\Delta}^M v_V^2 = 2|A|^2 v_V^2 + 6|\nabla v_V|^2\,.
\end{align*}
\label{v2evolution}
\end{lem}
The support function is defined as $u:M\rightarrow\R$:
\begin{align*}
u(x) :&= \IP{x}{\nu^M(x)}\,.
\end{align*}
It is strictly positive for convex bodies (that contain the origin) and vanishes on linear subspaces of $\R^{n+1}$.
On a minimal hypersurface, it satisfies the following equation:
\begin{lem}The quantity $u^2:M\rightarrow\R$ satisfies
\begin{align*}
{\Delta}^M u^2 = - 2|A|^2u^2 + 2|\nabla u|^2\,.
\end{align*}
\label{g2evolution}
\end{lem}
We also require the following evolution equation for the product of $v_V^2$ and $u^2$.
\begin{lem}
The quantity $Q=u^2v_V^2:M\rightarrow\R$ satisfies
\begin{align*}
{\Delta}^M Q \geq 2\frac{\nabla v_V}{v_V}\nabla Q .
\end{align*}
\label{g2v2evolution}
\end{lem}
\begin{proof}
We compute using Lemmata \ref{v2evolution} and \ref{g2evolution}
\begin{align*}
{\Delta}^M Q &= v_V^2 {\Delta}^M u^2 + u^2{\Delta}^M v_V^2 + 2\nabla v_V^2 \nabla u^2
	\\
 &= 2 v_V^2 |\nabla u|^2 + 6u^2|\nabla v_V|^2 + 2\nabla v_V^2 \nabla u^2
	\,.
\end{align*}
Separately we can transform the mixed gradient term into a gradient of the $Q$
quantity and extra terms as follows:
\begin{align*}
2\nabla v_V^2 \nabla u^2 &= \nabla v_V^2 \nabla u^2 + 4uv_V\nabla v_V \nabla u
	= 2\frac{\nabla v_V}{v_V}\nabla Q - 4|\nabla v_V|^2 u^2 + 4uv_V\nabla v_V \nabla u\,.
\end{align*}
Replacing into the above completes the proof:
\begin{align*}
{\Delta}^M Q &= 2\frac{\nabla v_V}{v_V}\nabla Q + 2v_V^2|\nabla u|^2 +
	6u^2|\nabla v_V|^2 - 4|\nabla v_V|^2 u^2 + 4uv_V\nabla v_V
	\nabla u
	\\
 &= 2\frac{\nabla v_V}{v_V}\nabla Q + 2v_V^2|\nabla u|^2 + 2u^2|\nabla v_V|^2 + 4uv_V\nabla v_V \nabla u
	\\
 &\geq 2\frac{\nabla v_V}{v_V}\nabla Q\,.
\end{align*}
\end{proof}

The following calculation is standard.

\begin{lem}[Simon's Equation] The square of the second fundamental form satisfies the equation
\begin{align*}
\frac{1}{2}{\Delta}^{M}|A|^2 = |\nabla A|^2 - |A|^4\,.
\end{align*}
\label{Aevolution}
\end{lem}
We also require the boundary derivative of the second fundamental form.
For any boundary point $x\in \partial M$ the Neumann boundary condition allows
us to chose a basis  $\{{\tau}_1,\ldots,{\tau}_{n}\}$ of the tangent space $T_x
M$ such that ${\tau}_i\in T\partial M \cap T\S^n $ for all
$i\in\{1,\ldots,n-1\}$ and ${\tau}_n={\nu}^{\S^n}$ at $x$.

Note that for any choice of orthonormal basis of the tangent space of $\S^n$
the second fundamental form of $\S^n$ is diagonal.
It was shown by Stahl \cite{stahl1996convergence} that on the boundary we have
\begin{align*}
h_{in}=h^{\S^n}_{in}=0\quad \text{ for all }\quad i=1,\ldots,n-1.
\end{align*}
Since $n-1$ of the tangent vectors are on the closed submanifold $\partial M\subset\S^n$ we can choose
our basis above such that on the boundary
\begin{align*}
h_{ij}=0\quad  \text{ if } \quad i\neq j\,,\quad \text{ and }\quad h^{\S^n}_{ij}=\delta_{ij}\quad\text{ if }\quad i\neq j.
\end{align*}
We also need the following boundary relations, again due to Stahl \cite[Theorem 2.4]{stahl1996convergence}.

For the remainder of this section we additionally assume that $F$ has free boundary on
$\S^n\subset\R^{n+1}$.

\begin{lem}[Normal derivatives]
Along the boundary in an orthonormal basis as described above the following hold:
\begin{align*}
\nabla_n H &= H\,,
\\
\nabla_n h_{ii} &= (h_{nn}-h_{ii})\quad\text{ for all }\quad i=1,\ldots,n-1.
\end{align*}
\label{boundary_derivatives}
\end{lem}

We are now ready to state our result about the boundary derivative of the second fundamental form squared.

\begin{prop}[Normal derivative of $|A|^2$]
Along the boundary in an orthonormal basis as described above the following hold:
\begin{align*}
\nabla_n|A|^2 = -2|A|^2- 2nh_{nn}^2\,.
\end{align*}
\label{Aboundary}
\end{prop}
\begin{proof}
The result follows by using Lemma \ref{boundary_derivatives} and
\[
\displaystyle \nabla_n h_{nn}=\nabla_n H- \sum_{i=1}^{n-1}\nabla_nh_{ii}\,.
\]
\end{proof}

\section{Minimal graphical hypersurfaces with free boundary are flat disks}
\label{KillingSection}

Now assume that
\begin{align}
	s_V > 0 \label{graphcondition}\,.
\end{align}
We shall prove Theorem \ref{thmuniqueness}.

\begin{proof}
Apply the elliptic maximum principle to the evolution of $Q\geq 0$ to see that
there is no interior maxima. On the boundary we note that $Q=u=0$ using
$F=\nu^{\S^n}$ and the Neumann boundary condition.

Therefore $Q = u^2v_V^2 = 0$ everywhere on $M$.
By hypothesis $v_V \ne 0$ and so the support function vanishes on $M$.
Using now the smoothness assumption, non-flat cones are ruled out, leaving only
the possibility that $M$ is a flat disk.
\end{proof}

\begin{rmk}
It is only necessary to assume that $s_V>0$ on the interior, it may a-priori assume zeroes on the boundary.
\end{rmk}

\begin{rmk}
The bounded curvature and graphicality condition are required to bound the
coefficient of the gradient term in the evolution of $Q$.
\end{rmk}

\section{Curvature gap}

We require the isoperimetric equality stated in our introduction so we derive it here.

\begin{proof}[Proof of Proposition \ref{isoequality}]
This is a simple computation using the divergence theorem. We include
it here for completeness while noting that more general proofs and applications
can be found in \cite{brendle2012,federer2014geometric,hanes1972isoperimetric}.
We calculate
\begin{align*}
0 = \int_{M}\IP{-H\nu^M}{F}d\mu
  = \int_{M}\IP{\Delta F}{F}d\mu
  = -\int_{M}|\nabla F|^2 + \int_{\partial M}\IP{\nabla_{\nu^{\S^n}}F}{F} dS\,,
\end{align*}
where we have used minimality in the first equality, the evolution of the
position vector in the second and the divergence theorem in the last. We denote
the unit outer normal to $\S^n$ by $\nu^{\S^n}$. Due to the perpendicular
boundary condition of the minimal hypersurface we also have
$\nu^{\S^n}=\nu^{\partial M}$.
Note that $\nu^{\S^n}=F$ on $\partial M\subset \S^n$, $|\nabla F|^2=n$, and that
$\nabla_{\nu^{\S^n}}F|_{\partial M^n}=F$ giving us that
$\IP{\nabla_{\nu^{\S^n}}F}{F}=1$ on $\partial M\subset \S^n$. This completes our
proof.
\end{proof}

We are now ready to prove Theorem \ref{thmuniqueness2}.

\begin{proof}[Proof of Theorem \ref{thmuniqueness2}]
  We will use the divergence theorem, this time on the second fundamental form squared.
	Assume for the moment that
\begin{equation}
               \bigg(\sup_{M} |A|^2\bigg)^2 \le n\inf_{\partial M} |A|^2\,.
\label{EQcontradiction}
\end{equation}
Now compute
\begin{align*}
\int_M \Delta |A|^2 d\mu = \int_{\partial M} \nabla_{\nu^{\S^n}} |A|^2 dS
= -\int_{\partial M} 2|A|^2 dS  - \int_{\partial M} 2n h_{nn}^2dS\,,
\end{align*}
where we have used Proposition \ref{Aboundary} in the last equality.
If $n=2$ then $2n h_{nn}^2 = n|A|^2$ and we find
\begin{align*}
\int_M \Delta |A|^2 d\mu 
= -4\int_{\partial M} |A|^2 dS\,.
\end{align*}
If $n>2$, we are unable to use the last term at all since nothing is preventing
$h_{nn}$ from vanishing even though the full $|A|^2$ does not.
In this case we find
\begin{align*}
\int_M \Delta |A|^2 d\mu 
\le -2\int_{\partial M} |A|^2 dS\,.
\end{align*}
In either case, Lemma \ref{Aevolution} implies
\begin{align*}
\int_M |\nabla A|^2d\mu &=  \int_M |A|^4d\mu +  \frac{1}{2}\int_M \Delta |A|^2 d\mu \\
&\leq \int_M |A|^4d\mu  - \int_{\partial M} |A|^2 dS
\leq |M|\bigg(\sup_{M} |A|^2\bigg)^2 - |\partial M|\inf_{\partial M} |A|^2
\,.
\end{align*}
Now \eqref{EQcontradiction} combined with Proposition \ref{isoequality} yields
\begin{align*}
\int_M |\nabla A|^2d\mu \le 0\,.
\end{align*}
This gives us that $|\nabla A|^2\equiv 0$ everywhere on $M$, implying that all
principal curvatures of $F$ are constant. This implies the second fundamental
form squared is constant and using the boundary derivative from Proposition
\ref{Aboundary} we see that $|A|^2\equiv 0$. Thus $M$ is a part of a plane. The
only plane with perpendicular boundary condition are equatorial disks.
But this is a contradiction with our hypothesis that $M$ is not a
standard flat disk.

Therefore the assumption \eqref{EQcontradiction} is false, and we conclude the
result of the theorem.
\end{proof}

\section*{acknowledgements}

The authors would like to thank Frank Morgan for his interest in this small contribution and for helping point out the correct references. The second author is supported by ARC grant Discovery grant DP150100375 at the University of Wollongong.

\bibliographystyle{plain}
\bibliography{mbib}

\end{document}